\documentclass[12pt,fleqn,leqno]{article}

\usepackage[english]{babel}

\usepackage[letterpaper,top=1in,bottom=1in,left=1in,right=1in]{geometry}

\usepackage{amsmath}
\usepackage{amssymb}
\usepackage{graphicx}
\usepackage{amsthm}
\usepackage{bm}
\usepackage{array}
\usepackage{overpic}
\usepackage{enumitem}
\usepackage[colorlinks=true, allcolors=blue]{hyperref}
\usepackage{cleveref}
\usepackage{comment}
\usepackage{subcaption}
\usepackage{wasysym} 

\newcommand{\R}{\mathbb R}
\newcommand{\C}{\mathbb C}
\newcommand{\E}{\mathbb E}
\newcommand{\cala}{\mathcal A}
\newcommand{\cald}{\mathcal D}
\newcommand{\calf}{\mathcal F}

\newcommand{\calk}{\mathcal K}
\newcommand{\caln}{\mathcal N}
\newcommand{\calo}{\mathcal O}
\newcommand{\calr}{\mathcal R}
\newcommand{\calv}{\mathcal V}
\newcommand{\calx}{\mathcal X}

\newcommand{\eps}{\varepsilon}
\newtheorem{theorem}{Theorem}[section]
\newtheorem{proposition}[theorem]{Proposition}
\newtheorem{corollary}[theorem]{Corollary}

\newtheorem{remark}[theorem]{Remark}
\newtheorem{example}[theorem]{Example}
\newcommand{\mtxa}[2]{\left[\!\!
 \begin{array}{#1} #2 \end{array} \!\! \right]}

\newcommand{\wh}{\widehat}
\newcommand{\wt}{\widetilde}

\DeclareMathOperator*{\argmin}{arg\,min}

\usepackage{xcolor}

\renewenvironment{abstract}
 {\small
 \begin{center}
 \bfseries \abstractname\vspace{-.5em}\vspace{0pt}
 \end{center}
 \list{}{%
 \setlength{\leftmargin}{5mm}% <---------- CHANGE HERE
 \setlength{\rightmargin}{\leftmargin}%
 }%
 \item\relax}
 {\endlist}

\title{Transpose-free linear algebra\footnote{Date \today}}
\author{Diana Halikias\footnote{Courant Institute, New York University, New York, NY 10012, USA, \texttt{diana.halikias@nyu.edu}.} \and
Michiel E.~Hochstenbach\footnote{Department of Mathematics and Computer Science, TU Eindhoven, The Netherlands, \texttt{m.e.hochstenbach@tue.nl}.} \and Alex Townsend\footnote{Mathematics Department, Cornell University, Ithaca, NY 14853-4201, USA, \texttt{townsend@cornell.edu}.}}
\date{}

\begin{document}
\maketitle

\begin{abstract}
We study the limitations of matrix-free algorithms that access a matrix $A$ only through forward matrix-vector products (matvecs) $x \mapsto Ax$, without access to the transpose $A^\top$ or its action. This setting arises naturally in operator learning, inverse problems, and matrix-free PDE solvers, where adjoint evaluations may be unavailable or prohibitively expensive. We show that the lack of transpose access creates severe and sometimes insurmountable theoretical barriers. For Krylov methods, we prove that the sequence of projected operator norms produced by Arnoldi iteration can follow any prescribed nondecreasing curve, showing that forward matvecs alone provide essentially no reliable information about the spectral norm. For several core problems, including least squares, norm estimation, column subset selection, and local maximum volume, we establish non-identifiability results; distinct matrices can generate identical forward-query transcripts while having fundamentally different solutions. We also prove quantitative lower bounds on the number of forward matvecs required for approximation tasks. In particular, any algorithm that computes a near-optimal rank-$k$ approximation must use at least $n$ queries, and estimating the Frobenius norm to relative accuracy $\eps$ requires $\Omega(\eps^{-2})$ queries when $n$ is sufficiently large, matching the complexity of Hutchinson-type estimators up to constants. Although some problems remain solvable without transpose access, the transpose-free setting is fundamentally more limited in both identifiability and efficiency.
\end{abstract}

\section{Introduction} \label{sec:intro}
In modern numerical linear algebra, it is common to treat a matrix $A$ as a black box that can be queried through matrix-vector products. This matrix-free viewpoint underpins Krylov subspace methods, randomized linear algebra, and large-scale scientific computing more broadly, especially when forming or storing $A$ explicitly is infeasible. In many applications, however, access to the transpose or adjoint operator is unavailable, inaccurate, or significantly more expensive than forward evaluations. This occurs in matrix-free discretizations of differential equations~\cite[p.~52]{ChanDePillisVdV98}, operator learning from input-output data~\cite{boulle2024operator,boulle2024mathematical}, matrix recovery~\cite{halikias2024structured,SunWoodruffYangZhang:2021}, and inverse problems in which the adjoint corresponds to a different physical model. These examples raise a basic question: what can be computed about a matrix when one has access only to the map $x \mapsto Ax$, but not to $y \mapsto A^\top y$?

We study this question in a query model. The matrix $A \in \R^{m \times n}$ is accessed through an oracle that returns $Ax$ for any query vector $x \in \R^n$. An algorithm may choose queries $x_1,\dots,x_q$ adaptively, with each $x_i$ depending on previous queries and responses and possibly on internal randomness. After observing the transcript $(x_1,Ax_1),\dots,(x_q,Ax_q)$, the algorithm outputs an estimate of a target quantity, such as a norm, a singular vector, or a low-rank approximation. We call an algorithm \emph{transpose-free} (TF) if it accesses $A$ only through such forward queries and has no access to $A^\top$ or to the entries of $A$ itself. Our main object of study is the number of forward queries needed to solve a given problem to prescribed accuracy and success probability. This setup is intentionally minimal: we do not restrict the algorithmic paradigm, the amount of internal computation, or the adaptivity of the queries. The only resource we measure is access to the forward oracle.

This model lets us separate two distinct issues. The first is \emph{identifiability}: whether an efficient set of forward queries uniquely determines the quantity of interest at all. The second is \emph{query complexity}: how many forward matvecs are required to solve a problem to a given accuracy. This distinction turns out to be useful throughout the paper. Many problems fail at the level of identifiability; the same forward information is consistent with multiple incompatible answers arising from different, indistinguishable matrices. For some problems, we prove a query complexity lower bound showing that regardless of the transcript, a problem may be solvable with only forward matvecs, but only at a query cost that is substantially larger than what is possible when both $A$ and $A^\top$ are available.\footnote{We consider algorithms that use $\mathcal O(n)$ forward queries, where $n$ is the number of columns in the matrix, to be  inefficient. This contrasts with the  limited memory setting, which requires $\Omega(n)$ matvecs~\cite{BLW25}.} Our main message is that TF access is fundamentally limited in what can be computed efficiently compared to the standard matrix-free access setting with both $A$ and $A^\top$ available.

Our results fall into three categories. First, we prove non-identifiability results for several  problems, including least squares, maximum volume, and aspects of spectral estimation. Given a ground-truth matrix $A$ and a transcript, we construct an adversarial matrix $B$ with an identical transcript, but yielding a very different solution. These are strong impossibility results in the sense that no amount of computation can recover information that is absent from the transcript. Second, we prove lower bounds on TF query complexity, independent of the transcript. In particular, a near-optimal rank-$k$ approximation of an $n \times n$ matrix requires $n$ forward queries, and estimating the Frobenius norm to relative accuracy $\eps$ requires $\Omega(\eps^{-2})$ forward queries. Third, we show that TF Krylov information can be highly misleading. The projected norms produced by Arnoldi can follow any prescribed nondecreasing curve, so these quantities do not reliably track the spectral norm. These three types of results do not immediately imply one another, and indeed  sometimes yield very different lower bounds. Thus, they each provide different qualitative descriptions of the limitations of the TF setting. 

TF computation has a long history in numerical linear algebra, especially in the development of Krylov methods.
 Much of the existing literature focuses on algorithm design under one-sided access. Our goal here is complementary. Instead of proposing yet another TF algorithm, we ask what can and cannot be inferred from one-sided information in the first place, and how expensive successful inference must be when it is possible.
 
The rest of the paper is organized as follows. In~\Cref{sec:review}, we review some existing algorithms and theoretical results for TF linear algebra. In~\Cref{sec:background}, we describe the techniques used to prove our original results in the TF setting. In \Cref{sec:spectral_norm_est}, we study TF approximation of the spectral norm, showing that using Krylov queries, projected norms can follow arbitrary nondecreasing curves in~\Cref{sec:krylov} and that the solution is not identifiable unless one uses $n$ queries in~\Cref{sec:spectral_norm_identifiability}. In \Cref{sec:low-rank}, we prove a sharp lower bound on the number of queries needed for near-optimal low-rank approximation and show that the set of columns spanning the range of the matrix is not identifiable unless one uses $n$ queries.  In \Cref{sec:css}, we establish non-identifiability for local maximum volume. In \Cref{sec:schatten}, we study Schatten-norm estimation, and prove a query complexity lower bound for estimating the Frobenius norm, resulting in a rank-one corollary for all Schatten norms. In \Cref{sec:least-squares}, we prove a non-identifiability result for overdetermined least squares problems with orthonormal design matrices. We conclude with open problems in \Cref{sec:conclusions}.

\medskip\noindent
{\bf Notation:} Unless otherwise specified, $\|\cdot\|$ denotes the spectral norm for matrices and the Euclidean norm for vectors. The Frobenius norm is denoted by $\|\cdot\|_F$.
The matrix $A$ is $m \times n$ or $n \times n$, depending on the problem.
We denote by $A_k$ the best rank-$k$ approximation to $A$ given by the Eckart--Young theorem.
The matrix trace is denoted by $\operatorname{tr}(A)$. 
The identity matrix of appropriate dimension is denoted by $I$, and its $i$th column by $e_i$.

\subsection{Main contributions}
\noindent We now describe our main technical contributions. We summarize our theoretical contributions, along with existing results, in~\Cref{tab:comparison}.
\begin{enumerate}[leftmargin=*]
 \item \textbf{TF Krylov information can be arbitrarily misleading.} In \Cref{prop:svd}, we show that for any prescribed positive nondecreasing sequence $0 < \sigma_1 \le \cdots \le \sigma_n$, there exists a matrix $A$ for which the Arnoldi projected norms satisfy $\|H_{k+1,k}\| = \sigma_k$ for $1 \le k \le n-1$, while $\|A\| = \sigma_n$. Thus, the projected norms seen by a TF Krylov method may not reveal any meaningful information about the spectral norm. In \Cref{prop:eig}, we strengthen this by showing that the same pathology can occur even after fixing the eigenvalues of $A$.
 \item \textbf{Near-optimal low-rank approximation is matvec-expensive without the transpose.} \Cref{thm:lra-lower-bound} shows that any randomized, possibly adaptive TF algorithm that returns a $(1+\eps)$-optimal rank-$k$ approximation for every $n \times n$ input matrix must use $n$ forward matvecs. In particular, one cannot obtain a genuinely matrix-free low-rank approximation algorithm from forward products alone.
 \item \textbf{Norm estimation exhibits sharp TF lower bounds.} For Frobenius norm estimation, \Cref{thm:frobenius-lower-bound} proves that relative-error estimation within $\eps$ requires $\Omega(\eps^{-2})$ forward queries, for sufficiently large $n$, matching Hutchinson-type methods up to constants and ruling out Hutch++-type $\calo(\eps^{-1})$ improvements in the TF model. In particular, Hutchinson's estimator for Frobenius norm estimation is optimal among TF methods. As an immediate consequence, \Cref{cor:schatten-rank-one} gives the same lower bound for estimating any Schatten norm, as all Schatten norms coincide on rank-one matrices.
 \item \textbf{Several basic problems are not identifiable from one-sided transcripts.} In \Cref{prop:css_identifiability}, we show that  a  set of columns spanning the range of the matrix is not identifiable without $n$ forward matvecs. In \Cref{prop:css-near-local-maxvol-nonidentifiability}, we prove that a solution to the local maximum volume problem for an $m \times n$ matrix is also not determined by fewer than $n$ forward matvecs. In \Cref{prop:ls-nonuniqueness}, we prove a non-identifiability result for overdetermined least squares problems with orthonormal  matrices, even when the adversarial matrix has condition number 1. Finally,  a simple adversarial construction in~\Cref{prop:spectral_norm_identifiability} establishes non-identifiability for  norm estimation. 
 \item \textbf{A review of the TF literature.}
While TF methods have existed for a long time, the present work provides a unified perspective on this literature. Beginning with the historical origins and motivations for the TF setting, we discuss methods for large-scale linear systems, ill-posed problems, and  model reduction in \Cref{sec:review}. We conclude by describing operator learning, as well as the closely related ``unmatched transpose'' setting. 
\end{enumerate}

\newcolumntype{L}[1]{>{\raggedright\arraybackslash}p{#1}}

\begin{table}[htbp]
\centering
\footnotesize
\renewcommand{\arraystretch}{1.2}

\begin{tabular}{L{3.5cm} L{6.3cm} L{5cm}}
\hline
Problem & Transpose-free results & With $A$ and $A^\top$ \\
\hline

Spectral norm estimation &
\begin{itemize}[leftmargin=*, nosep] \vspace{-2mm}
\item Arbitrary projected-norm curves using Krylov inputs (Props.~\ref{prop:svd}, \ref{prop:eig})
\item $n$ queries for identifiability (Prop.~\ref{prop:spectral_norm_identifiability})
\end{itemize}
&
Standard Krylov estimators, randomized methods \\ \\

Rank-$k$ approximation (constant-factor or $(1 + \eps)$-optimal)&
\begin{itemize}[leftmargin=*, nosep] \vspace{-3mm}
\item $n$ query complexity (Thm.~\ref{thm:lra-lower-bound})
\item $n$ queries for identifiability~\cite[Lem.~2.3]{halikias2025structured}
\end{itemize}
&
Efficient randomized methods: $\calo(k \, \eps^{-1})$ queries \\ \\

Rank-$k$ approximation via column subset selection &
\begin{itemize}[leftmargin=*, nosep] \vspace{-3mm}
\item cannot form $CC^\dagger$ approximation without the transpose
\item $n$ queries for identifiability of  $k$ cols. of $A$ that span $\text{col}(A)$ (Prop. \ref{prop:css_identifiability})
\end{itemize}
&
Standard column access from queries with $A$ or randomized queries; project $A$ onto the subspace spanned by the chosen columns by querying $A^\top$. \\ \\

Frobenius / Schatten norm estimation &
\begin{itemize}[leftmargin=*, nosep] \vspace{-3mm}
\item $\Omega(\eps^{-2})$ query complexity for $n$ sufficiently large (Thm.~\ref{thm:frobenius-lower-bound}, Cor.~\ref{cor:schatten-rank-one}); matches complexity of Hutchinson's estimator for $\|A\|_F^2$ 
\item $n$ queries for identifiability (Prop.~\ref{prop:spectral_norm_identifiability})
\item unbiased estimator for $\|A\|_{S^{2p}}$~\cite{kongvaliantspectrum2017}
\item $\Omega(n^{1 - 4/p})$ non-adaptive queries for even $p > 4$~\cite{linguyenwoodruffon2014}
\end{itemize}
&
Hutch++-type improvements for $\|A\|_F$: $\calo(\eps^{-1})$ queries \\ \\

Local maxvol &
\begin{itemize}[leftmargin=*, nosep] \vspace{-3mm}
\item $n$ queries for identifiability (Prop.~\ref{prop:css-near-local-maxvol-nonidentifiability})
\end{itemize}
&
Standard column access \\ \\

Least squares &
\begin{itemize}[leftmargin=*, nosep] \vspace{-2mm}
\item $n$ query complexity \cite[Thm.~1.2]{derezinski2026matrix}
\item $n$ queries for identifiability (Prop.~\ref{prop:ls-nonuniqueness})
\item GMRES convergence curve results~\cite{GPS96}
\end{itemize}
&
$Q^\top b$ available; only one query needed\\ \hline
\end{tabular}
\caption{Summary of some existing results on TF barriers and the contributions in this work. }
\label{tab:comparison}
\end{table}

\section{A brief review of transpose-free methods}
\label{sec:review}
Many classical algorithms for nonsymmetric problems, such as bi-Lanczos-based methods and their descendants, rely on access to both $A$ and $A^\top$. By contrast, Arnoldi-based methods such as GMRES~\cite{SaadSchultz86} use only forward products. A substantial literature from the late 1980s and 1990s sought to remove transpose access while retaining short recurrences, leading to methods such as CGS~\cite{Sonneveld89}, BiCGSTAB~\cite{vanderVorst92}, and TFQMR~\cite{Fre93}, as well as TF reformulations based on polynomial and squaring techniques~\cite{ChanDePillisVdV98}. The broader literature also contains ``anything is possible'' pathology results for Krylov methods, most famously for GMRES~\cite{GPS96}; see also \cite{DTe17,Sch16,TMe12,VLa11}. For spectral norm estimation, however, standard algorithms typically exploit the transpose, e.g., via Lanczos bidiagonalization or Krylov--Schur methods applied to $A^\top \! A$~\cite{BRe05,Sto12}, and even rough estimators often use iterations with both $A$ and $A^\top$. More recently, the TF setting has appeared in operator learning and data-defined operators~\cite{boulle2024operator,boulle2024mathematical}. In this section, we give a short review of TF algorithms and theory in numerical linear algebra. 

\subsection{Motivations for TF methods} \label{sec:motivation}
Traditionally, there are three contexts where TF methods are especially useful~\cite{ChanDePillisVdV98}. First, for the solution of nonlinear equations $F(x) = 0$, with $F: \R^n \to \R^n$, the action of the Jacobian $x \mapsto DF(a) \, x$ is often used in inner iterations for the solution of a large-scale linear system, as in \emph{Newton--Krylov methods}.
This action is then  approximated by a difference quotient $\delta^{-1} \, (F(a+\delta \, x)-F(a))$ for a small $\delta$ (typically, $10^{-8}$), while there is no similar expression for the transpose.
Second, the action of the forward operator may be implemented in a routine or software package, such as in integral equations \cite{ChanDePillisVdV98} or computed tomography (CT) \cite{EHa18}.
Finally, TF methods are natural for popular matrix storage types (e.g., compressed row format) which enable efficient matrix-vector products. The transpose may be  much more expensive, or require twice the storage (a compressed column format).

\subsection{Linear systems}
Consider the solution of a nonsingular linear system $Ax = b$, where $A \in \R^{n \times n}$.
A  variety of Krylov methods (see, e.g., \cite{templates94} for an overview) 
 search for an approximate solution in the Krylov spaces $\calk_k(A,b) := {\rm span}\{b, Ab, \dots, A^{k-1}b\}$.
Some methods are explicitly designed for symmetric $A$, such as CG, MINRES, and SYMMLQ.
Some methods for nonsymmetric $A$, such as GMRES and FOM, do not need the action of the transpose,  at the price of storing  basis vectors and (re)orthogonalization against  those vectors. 
Two-sided methods, like BiCG, use the transpose.
Short recurrences enable reduction in both the storage of basis vectors and  (re)orthogonalization costs.
To exploit a shadow Krylov space $\calk_k(A^\top, c)$ for a shadow vector $c$, these methods use vectors of the form $\psi(A^\top) \, c$, where $\psi$ is a degree-$k$ polynomial.
 However, the convergence of BiCG is typically quite irregular and non-monotonic.

CGS, BiCGSTAB, and the TFQMR methods seek  to improve on the BiCG approach, first by stabilizing the convergence. These methods are also suited to avoiding transposes.
The essential trick  is to access these vectors only through inner products, as 
$(\psi(A^\top) \, c, \ \varphi(A) \, b) = (c, \ \psi(A) \, \varphi(A) \, b)$.  Rather than working with a residual and its conjugate (as in BiCG), CGS uses  a squared residual, which also avoids  the transpose.
TFQMR is  related to CGS and has the advantage of avoiding irregular convergence behavior by quasi-minimizing a residual norm.
Both BiCGSTAB and TFQMR are widely used  methods  in science and engineering.\footnote{In fact, \cite{vanderVorst92} was the most cited math paper of the 1990s.}

Various other methods following a similar strategy have been proposed.
BiCGSTAB($\ell$) \cite{SFo93} uses polynomials of degree $\ell$ instead of one for the shadow residual to get an even smoother convergence.
A lesser-known variant is QMRCGSTAB \cite{CHS94}, which is a QMR type adaptation of BiCGSTAB, in the same way that TFQMR is of CGS.
Some other variations are discussed in \cite{BRZ98}.
A relatively recent TF linear solver, based on an idea of the 1980s, is IDR($s$) \cite{SGi09}.
On two-sided methods, it has been shown in 
\cite{FMa84} that in a certain class of methods exploiting a three-term recurrence, none minimizes the residual norm; see also \cite{DTT98}.
A review on developments of Krylov methods in the 21st century can be found in \cite{SSz07}.

TF methods have been proposed for several generalized problems. For the block case $AX = B$, where $B$ comprises many right-hand sides, \cite{SGa95} presents a hybrid TF block method, which is claimed to use less memory and perform better than a TF block GMRES method.
The TFQMR method can also be elegantly used for families of shifted linear systems \cite{Fre93b}.

Finally, the computation of matrix functions,  $x = f(A) \, b$, may be viewed as a generalization of linear systems, which are obtained by  $f(A) = A^{-1}$.
Several aspects of TF subspace extraction methods are compared with those of two-sided approaches (using $A^\top$) in \cite{HHo05}.

In the realm of TF lower bounds for linear systems,~\cite{derezinski2026matrix} recently established a query complexity lower bound of $n$ forward queries. In the present work, we establish an identifiability lower bound of $n$ forward queries for orthonormal  least squares problems.

\subsection{Linear ill-posed problems}
Linear ill-posed problems are of the form $Ax \approx b$, with $A \in \R^{m \times n}$, where all three situations $m=n$, $m > n$, and $m<n$ occur frequently in practice, and $A$ is (near) singular.
To get sensible approximate solutions, it is common to add a regularization term; a popular alternative is to use an iterative method as regularization, by stopping early enough in the process.
The iteration number $k$ then acts as a regularization parameter.

For square problems, both the TF method GMRES and the non-TF approach LSQR are popular.
There is no single best method for all problems. However, as reported by experts and practitioners, the method of choice for a wide range of problems is  LSQR \cite{Han25}.
One heuristic reason for this is that small singular values play an important role for these problems and are approximated well by  LSQR Krylov spaces of the form $\calk_k(A^\top \! A, \, A^\top b)$, while GMRES Krylov spaces $\calk_k(A, b)$ by nature target eigenvalues.
Indeed,  the fact that the default method, LSQR, queries $A^\top$ is in line with the general message of this paper.

For the overdetermined case ($m > n$), only methods involving Krylov spaces with $A^\top$ (such as LSQR, LSMR, CGLS, CGNR, CGNE, and SYMMLQ), seem to be natural options.
An alternative to iterative methods as a regularizer is to add a regularization term involving a 2-norm, 1-norm, or other $p$-norm.
For these approaches, flexible Krylov methods are popular, including a variable right preconditioner to obtain the desired norm regularization.
The main underlying mechanisms are flexible Arnoldi (which is TF) and flexible Golub--Kahan (Lanczos) bidiagonalization (exploiting $A^\top$); see, e.g., \cite{CGa19}.

\subsection{Model reduction of LTI dynamical systems}
A standard form of a (single-input, single-output) linear time-invariant (LTI) dynamical system is
$\dot x(t) = A \, x(t) + b \, u(t)$, \ $y(t) = c^\top x(t)$.
Here $b$, $c \in \R^n$, $u: \R \to \R$ is the input function, and $y: \R \to \R$ is the output function.
For sufficiently large $|s|$ so that the series converges, the associated transfer function has the expansion
\begin{equation} \label{eq:mor1}
H(s) = c^\top (s\,I-A)^{-1} \, b = s^{-1} \, (c^\top \,b) + s^{-2} \, (c^\top \!A\,b) + s^{-3} \, (c^\top \!A^2\,b) + \cdots
\end{equation}
Here, $s \in i \, \R$ indicates the frequency.
The terms $s^{-j} \, c^\top \!A^{j-1}b$ are the \textit{moments}.
Since $n$ is often large, dimension reduction is desirable; see, e.g., \cite{Ant05} for an overview of the contents of this subsection.
A breakthrough of the 1990s \cite{FFr95} shows that an accurate reduction of the transfer function can be obtained by projecting onto Krylov spaces, rather than explicitly computing moments.
One default technique is to use two-sided projections, including the transpose.
Let the columns of $V$ and $W$ span $\calk_k(A, b)$ and $\calk_k(A^\top, c)$, respectively.
For simplicity,  we assume that the bases are biorthogonal ($W^\top V = I$); minor adaptations are needed in case both $V$ and $W$ have orthonormal columns.
Then, the reduced transfer function
\[
\wh H(s) = (c^\top V) \, (s\,I-W^\top\!AV)^{-1} \, W^\top b
\]
matches the first $2k$ moments in the expansion of the transfer function.

An alternative to this two-sided approach is a TF one-sided projection, using only $V$ with orthonormal columns (so $V^\top V = I$).
The corresponding approximation
\begin{equation} \label{eq:mor2}
\wt H(s) = (c^\top V) \, (s\,I-V^\top\!AV)^{-1} \, V^\top b 
\end{equation}
matches only $k$ moments. This usually means that the approximation is slightly worse, especially for the higher frequencies.
However, such a TF approach may have other advantages over a two-sided scheme.
Most notably, a one-sided projection \eqref{eq:mor2} preserves stability if the field of values $\calf(A) := \{ \, x^\top Ax \, : \, x \in \R^n, \, \|x\| = 1 \, \}$ is located in the open left-half plane, since this implies that this also holds for the spectrum of $V^\top \! AV$ for any $V$ with orthonormal columns.
In this case, the eigenvalues of $V^\top\!AV$ are necessarily in the left-half plane.

Although \eqref{eq:mor1} represents a Taylor expansion at $s=\infty$, matching the first couple of moments by a projection often also yields good fits for smaller  frequency values.
However, for better fits near a given frequency $s_0$, a common approach is to consider the shifted transfer function
$c^T ((s-s_0)\,I - (A-s_0 \, I))^{-1} \, b$ instead, leading to moments of the form $c^T (A-s_0 I)^{-j} \, b$.

For generalized linear systems $E \, \dot x(t) = A\,x(t) + b \, u(t)$, with $E$ nonsingular, the associated transfer function is of the form
\[
H(s) = c^\top (s\,E-A)^{-1} \, b = c^\top (s\,I-E^{-1} A)^{-1} \, E^{-1}b,
\]
and there are similar one-sided and two-sided methods available as for the standard case.

\subsection{Operator learning}

Increasingly, scientists and engineers use deep learning methods to build reduced order models of dynamical systems and solution operators of PDEs. Neural operators~\cite{boulle2024mathematical, kovachki2024operator} are infinite-dimensional analogues of neural networks used to approximate these operators using data generated from physical experiments or numerical simulations. For linear solution operators, a natural idealized model accesses both forward queries $f \mapsto \cala f$ and adjoint queries $g\mapsto \cala^\ast g$. In most practical settings, however, only $f \mapsto \cala f$ is available.

If $\cala$ is the solution operator of a linear PDE, then access to the adjoint $\cala^\ast$ is usually unrealistic as the adjoint is the solution operator to a different PDE. Thus, the goal of recent work~\cite{boulle2024operator} is to build a theoretical framework for adjoint-free operator learning.  The central contribution of this work is an adjoint-free operator learning algorithm whose convergence relies on prior knowledge about ${\rm range}(\cala^\ast)$ encoded in a prior, self-adjoint operator $L$. In the case of elliptic PDEs, elliptic regularity is used to obtain an  algebraic adjoint-free convergence rate~\cite[Section 4.2]{boulle2024operator}. For more general classes of operators,  larger sample complexity  guarantees  that depend on the smoothness of the operator have been derived~\cite{adcock2024sample,adcock2025towards,kovachki2024data}.

Thus, while adjoint-free operator learning may be impossible or highly inefficient for general operators, operators arising from PDEs or dynamical systems may have structure that enables adjoint-free operator learning.  With adjoint access, existing convergence guarantees for learning elliptic PDEs are exponential, rather than algebraic~\cite{boulle2023elliptic}. It is an open question whether one can close this theoretical gap between the adjoint and adjoint-free settings.

Motivated by the operator learning setting, the analogous problem of matrix recovery from matrix-vector products has also been studied in many works~\cite{amsel2025query,amsel2026fixed, ChenDumanKelesHalikiasMuscoMuscoPersson:2025,halikias2024structured}.
In~\cite{amsel2025query}, a theoretical gap is established between the query complexity of matrix recovery using one-sided versus two-sided queries. To find  a near-optimal approximation within a finite family $\calf$ to a general matrix, they establish a quadratic improvement in query complexity over the TF setting. The algorithm uses sketches on one side to identify the fraction of ``bad'' approximations that sketches on the other side can eliminate from $\calf$. This barrier may  extend to other matrix recovery problems.  

\subsection{Unmatched transpose}
The previous subsections  focus  on TF methods for various problems.
There has also been considerable  attention to problems involving an operator with an \emph{unmatched transpose}, or \emph{adjoint mismatch}.
This usually arises from the fact that  the forward and transpose operators correspond to separate and extensive software implementations, where the resulting actions are not close to  transposes of one another.
A typical motivation comes from CT, see, e.g., \cite{ZGu00}, where $A$ corresponds to the projection and $A^\top$ to the backprojection.
An unmatched transpose may be viewed as having \textit{inexact information about $A$'s row space}.

Consider $A$ with unmatched transpose $B$.
Typically, $B$ is a ``quite inaccurate transpose of $A$,'' i.e.,  $BA$ is  a nonsymmetric operator with, besides positive eigenvalues, real but negative eigenvalues, complex eigenvalues with positive real part, and also complex eigenvalues with negative real part; see, e.g., \cite{DHH19}.
The operator $BA$ gives rise to interesting recent developments.
In \cite{EHa18, LRS18}, column- and row-oriented iterative methods are studied.
BA-GMRES is the topic of \cite{HHM22}. 
For a least squares problem $Ax \approx b$, it is shown in~\cite{Wat26} that ``not-normal equations'' of the form $BAx = Bb$ may have conditioning advantages compared to the usual normal equations.
In \cite{BLW25}, the authors address the question of how to quantify the norm of the mismatch.
A proximal gradient method for an ill-posed problem with a nonsmooth regularization term involving an unmatched transpose is discussed in \cite{CPS23}.

\section{Query model and identifiability}
\label{sec:background}

We now introduce the TF query model and the basic mechanisms that underlie the impossibility and lower-bound results in the rest of the paper. We use an oracle model closely related to standard models in matrix sketching and matrix recovery~\cite{halikias2024structured,SunWoodruffYangZhang:2021,woodruff2014sketching}. An unknown matrix $A \in \R^{m\times n}$ is accessed only through the operation
\[
x \longmapsto Ax, \qquad x \in \R^n.
\]
At round $i$, the algorithm chooses a query vector $x_i$, possibly adaptively and using internal randomness, and receives the response $Ax_i$. After $q$ rounds, it has observed the transcript
\[
(x_1,Ax_1), \dotsc, (x_q,Ax_q),
\]
and must output the object of interest. The central structural fact is that the transcript depends only on the action of $A$ on the query space
\[
\calx_q = \mathrm{span}\{x_1,\dots,x_q\}.
\]
The action of $A$ on anything outside of $\calx_q$ is invisible to the algorithm. This leads to two complementary notions that recur throughout the paper. First, a solution to a problem may be \emph{non-identifiable} from the transcript in the sense that for every matrix $A$, there exists a matrix $B$ such that
\[
Ax = Bx \qquad \text{for all } x \in \calx_q,
\]
even though the solutions to the problems involving $A$ and $B$ are significantly different. In this type of result, the construction of the adversary $B$ depends on the transcript. As a consequence, no deterministic TF algorithm can solve the problem from the transcript alone. Second, while a randomized algorithm may avoid the adversary with high probability, it may still suffer from high \textit{query complexity}; one needs many queries before the transcript contains enough information to guarantee success.

 Importantly, a lower bound on query complexity does not imply a lower bound on the number of queries needed for identifiability, nor vice versa. For example, in~\Cref{prop:spectral_norm_identifiability}, we show that for any matrix $A$, one needs $n$ queries with $A$ for the Frobenius norm of $A$ to be uniquely determined. However,  Hutchinson's estimator applied to $A^\top A$ needs only $\calo(\eps^{-2})$ queries to achieve an $\eps$-approximation. This is because given a query space $\mathcal X_q$, a  non-identifiability argument constructs an adversary using  $\mathcal X_q^\perp $. Randomized algorithms can circumvent this obstruction because  $\mathcal X_q$ itself is random, and a fixed adversary is unlikely to remain indistinguishable from $A$ across all random query realizations. Conversely, query complexity lower bounds obtained via Yao's minimax principle do not explicitly characterize the adversarial matrices  associated with $\mathcal X_q$, whereas non-identifiability arguments  directly exhibit such matrices,  providing geometric insight into the problem. Thus, these types of results are qualitatively different measures of hardness and formally incomparable. 

\subsection{Yao's minimax principle and hard distributions}
\label{sec:yao}

Our quantitative lower bounds on query complexity are proved using Yao's minimax principle~\cite{Yao77}. In the present setting, the principle says that to lower-bound the query complexity of randomized, possibly adaptive TF algorithms, it suffices to find a hard distribution of matrices and analyze the average performance of all deterministic, adaptive algorithms over that distribution. Formally, if $\cala$ denotes a class of deterministic algorithms, $\calr$ the associated randomized algorithms, and $c(\mathrm{alg},A)$ a cost function, then
\[
\inf_{R \, \in \, \calr} \sup_{A} \ \E[c(R, A)]
\;\ge\;
\sup_{\cald} \inf_{\mathrm{alg} \, \in \, \cala} \ \E_{A \sim \cald}[c(\mathrm{alg}, A)].
\]
Thus, once a hard distribution $\cald$ is fixed, one may treat the algorithm as deterministic and focus entirely on what the transcript can and cannot reveal.

In Sections~\ref{sec:low-rank} and \ref{sec:schatten}, we use two kinds of cost functions. For  low-rank approximation, we take $c(\mathrm{alg},A)$ to be the failure indicator, so that its expectation is the failure probability. For Schatten-norm estimation, we instead take $c(\mathrm{alg},A)$ to be the estimation error. The  lower-bound arguments then follow a common pattern: construct a hard distribution, consider any deterministic algorithm with $q$ queries, and show that after observing the full transcript, there remains either ambiguity in the correct answer or enough residual uncertainty to force a nontrivial error. In this paper, we appeal to Yao's minimax principle in proving lower bounds on the TF query complexity of low-rank approximation and Frobenius/Schatten norm estimation. 

\subsection{The mechanism behind identifiability}

Several of our negative results are driven by the same geometric mechanism. We say that the solution to the problem is \textit{non-identifiable} if for all $A$, there exists a distinct matrix $B$ which satisfies the same matvecs as $A$, i.e., $AX = BX$, but the solutions to the problems with $A$ and $B$ are very different. The geometric insight is that one can ``hide'' directions orthogonal to $X$ in the matrix $B$ to produce a very different answer. This is formalized in the following proposition.

\begin{proposition}
 Let $A \in \R^{m \times n}$ be any matrix and $X \in \R^{n \times q}$ be any matrix with linearly independent columns. Then 
 \begin{align*}
 \{B \in \R^{m \times n} : BX = AX \} = \{A\,Q_1 Q_1^\top + Z Q_2^\top: Z \in \R^{m \times (n - q)}, \ Q_1 \in \R^{n \times q},\\ {\rm col}(Q_1) = {\rm col}(X), \ Q = [Q_1 \ Q_2] \ {\rm orthogonal}\}.
 \end{align*}
\end{proposition}
\begin{proof} 
Consider the QR factorization of the input matrix $X = Q_1 R$, where $Q_1 \in \R^{n \times q}$ has orthonormal columns and $R \in \R^{q \times q}$ is upper triangular. If $Q_2 \in \R^{n \times (n-q)}$ has orthonormal columns that are also orthonormal to the columns of $Q_1$, we note that the concatenation $Q \in \R^{n \times n}$ defined as $Q = [Q_1, Q_2]$ has orthonormal columns that span $\R^n$. Then, any $B$ satisfying $AX = BX \iff B\,Q_1 = A\,Q_1 \iff B\,Q_1 = AXR^{-1}$ also satisfies $B\,Q\,Q^\top = B$:
\[
B= B\,Q\,Q^\top = B\,Q_1\,Q_1^\top + B\,Q_2\!Q_2^\top = A\,Q_1\,Q_1^\top + \underbrace{B\,Q_2}_{Z} Q_2^\top,
\]
where $Z:= B\,Q_2 \in \R^{m \times (n-q)}$.
\end{proof}

Thus, determining if the solution of a linear algebra problem is identifiable from the transcript is equivalent to finding a $Z= B\,Q_2$ that results in a drastically different solution. For example, one may construct $Z$ so as to alter $A$'s columns lying outside of $\calx_q$, change the least squares solution by acting in a direction orthogonal to $\calx_q$, or plant additional norms or ranks in a direction orthogonal to $\calx_q$. This is the basic source of the non-identifiability results in later sections; the algorithm learns only a partial action of the matrix, and different global matrices can agree on that partial action.

\section{Spectral norm estimation}\label{sec:spectral_norm_est}

In this section, we consider  the problem of matrix-free  spectral norm estimation. With access to the transpose, there are both standard Krylov estimators and randomized methods well-suited to this task. However, we establish barriers to solving this problem in the TF setting. In particular, if one uses Krylov queries, we show that the Arnoldi projected norm approximations to the spectral norm can exhibit nearly arbitrary convergence. We also prove a non-identifiability result for TF spectral and Frobenius norm estimation.

\subsection{Krylov approximation of the spectral norm}\label{sec:krylov}
We begin with the Krylov setting. The results in this section should be viewed in the spirit of the ``any convergence curve is possible'' literature for GMRES and Arnoldi, beginning with the classical result of Greenbaum, Pt\'ak, and Strako\v{s} for GMRES convergence curves~\cite{GPS96} and followed by related pathology results for restarted GMRES, Arnoldi Ritz values, and harmonic Ritz values~\cite{DTe17,Sch16,TMe12,VLa11}. Those works show that, under surprisingly weak constraints, Krylov output can be made to follow nearly arbitrary prescribed behavior. Our result is analogous in flavor but different in target: instead of residual curves or Ritz values, we consider the Arnoldi projected norms used as TF proxies for the spectral norm.

Let $b$ of unit norm be given; we use this $b$ as the first column $v_1$ of $V_k$.
Recall that the Krylov space of order $k$ is given by
$\calv_k = \calk_k(A,v_1) := {\rm span}\{v_1, Av_1, \dots, A^{k-1}v_1\}$.
After $k$ steps, the Krylov relation is
\[
A V_k = V_{k+1} \, H_{k+1,k}.
\]
Here, $V_k$ and $V_{k+1}$ are matrices with orthonormal columns, spanning $\calv_k$ and $\calv_{k+1}$ respectively, and $H_{k+1,k}$ is an upper Hessenberg matrix.
In a TF setting, a natural Arnoldi approximation to the spectral norm is the norm of the restriction of the operator to $\calv_k$,
\[
\|A_{\vert_{\calv_k}}\| = \|AV_k\| = \|V_{k+1} \, H_{k+1,k}\| = \|H_{k+1,k}\|.
\]
(This quantity is at least as large as $\|H_{k,k}\| = \|V_k^\top \! AV_k\|$.)
The quantity $\|H_{k+1,k}\|$ is the natural object produced by Arnoldi from forward matvecs alone. The point of this section is that, just as GMRES and Arnoldi can exhibit highly noninformative prescribed behavior in the classical pathology literature, these projected norms can also behave in an essentially arbitrary way. In particular, forward-only Krylov information need not provide a reliable guide to $\|A\|$.

For the two results in this section, we use a technique that has also been used by \cite{GPS96}, in the first ``any convergence curve is possible'' paper: a companion matrix which is also upper Hessenberg.
The benefit of Hessenberg structure is that $\calk_k(A, e_1) = {\rm span}\{e_1, \dotsc, e_k\}$, while the companion matrix form makes it easier to prescribe the behavior of a specified desired quantity: an appropriate largest singular value (matrix two-norm) in our case. For consistency with the rest of the paper, our results are for  matrices with real valued, however  the identical proof also applies to complex matrices. 

\begin{proposition}[Any nondecreasing curve is possible; singular value version] \label{prop:svd}
Given a positive nondecreasing sequence $0 < \sigma_1 \le \cdots \le \sigma_n$ and a nonzero starting vector $b \in \R^n$,
there exists an $n \times n$ matrix $A$ such that $\|H_{k+1,k}\| = \sigma_k$, for $1 \le k \le n-1$ and $\|A\| = \sigma_n$.
\end{proposition}
\begin{proof}
Since Arnoldi starts from the normalized vector $b\,/\,\|b\|$, we may assume without loss of generality that $\|b\|=1$.
First consider the case $b = e_1$, the first standard basis vector.
Define the weighted cyclic shift $S \in \R^{n\times n}$ by
\[
S \, e_j = \sigma_j \, e_{j+1}, \qquad 1\le j\le n-1,
\qquad
S \, e_n = \sigma_n \, e_1.
\]
Note that $S$ is both upper Hessenberg and of companion matrix type.
Then $S^\top S = \operatorname{diag}(\sigma_1^2,\dots,\sigma_n^2)$, so the singular values of $S$ are exactly $\sigma_1,\dots,\sigma_n$, and therefore $\|S\| = \sigma_n$.

Applying Arnoldi to $S$ with starting vector $e_1$, the Krylov vectors are
\[
e_1, \quad S \, e_1 = \sigma_1 \, e_2, \quad S^2 \, e_1 = \sigma_1\sigma_2 \, e_3, \quad \dots,
\]
so the Arnoldi basis is the canonical basis up to harmless signs and scalings. Consequently, for $1\le k\le n-1$, the projected matrix $H_{k+1,k}$ has subdiagonal entries $\sigma_1,\dots,\sigma_k$ and zeros elsewhere. Its nonzero singular values are therefore $\sigma_1,\dots,\sigma_k$, and since the sequence is nondecreasing we obtain
\[
\|H_{k+1,k}\| = \max_{1 \, \le j \, \le \, k} \sigma_j = \sigma_k.
\]
Thus the proposition holds for the starting vector $e_1$.

For a general vector $b$ with unit norm, let $Q$ be any orthogonal matrix whose first column is $b$, and set $A = Q \, S \, Q^\top$.
Running Arnoldi on $A$ with starting vector $b$ is unitarily equivalent to running Arnoldi on $S$ with starting vector $e_1$. In particular, the Hessenberg matrices produced by the two processes are identical, so $\|H_{k+1,k}\| = \sigma_k$ for $1\le k\le n-1$. Since unitary similarity preserves the operator norm, $\|A\| = \|S\| = \sigma_n$.
\end{proof}
This proposition shows that the Arnoldi projected norms alone impose essentially no meaningful monotonic relation to the true spectral norm beyond the trivial bound $\|H_{k+1,k}\|\le \|A\|$. In particular, a forward-only Krylov method may observe an arbitrarily prescribed nondecreasing history before the final norm is revealed.

The point of the next result is that the pathology from \Cref{prop:svd} is not merely a consequence of unconstrained spectral data. Even after fixing the eigenvalues, the projected norms produced by Arnoldi can still be forced to follow an arbitrary nondecreasing curve.

\begin{proposition}[Any nondecreasing curve is possible; eigenvalue version] \label{prop:eig}
Given a positive and nondecreasing sequence $0 < \sigma_1 \le \cdots \le \sigma_{n-1}$, eigenvalues $\lambda_1$, \dots, $\lambda_n \in \C$, and a nonzero starting vector $b \in \R^n$,
there exists an $n \times n$ matrix $A$ such that:
\begin{itemize}
\item $\|H_{k+1,k}\| = \sigma_k$, \ for $1 \le k \le n-1$;
\item $A$ has eigenvalues $\{\lambda_1, \dots, \lambda_n\}$.
\end{itemize}
\end{proposition}
\begin{proof}
As in the proof of \Cref{prop:svd}, we may assume that $\|b\|=1$. Write
\[
\pi(\lambda) = \textstyle \prod_{j=1}^n \, (\lambda-\lambda_j)
= \lambda^n - \beta_1 \, \lambda^{n-1} + \beta_2 \, \lambda^{n-2} - \cdots + (-1)^n \, \beta_n,
\]
where $\beta_1, \dots, \beta_n$ are the elementary symmetric polynomials in $\lambda_1,\dots,\lambda_n$.

First consider the starting vector $b = e_1$. Define $A \in \R^{n\times n}$ to be the scaled companion type matrix
\[
A =
\mtxa{ccccc}{
0 & \cdots & \cdots & 0 & (-1)^{n-1} \, \frac{\beta_n}{\sigma_1 \, \cdots \, \sigma_{n-1}} \\[2mm]
\sigma_1 & 0 & & 0 & (-1)^{n-2} \, \frac{\beta_{n-1}}{\sigma_2 \, \cdots \, \sigma_{n-1}} \\[1mm]
0 & \sigma_2 & \ddots & \vdots & \vdots \\
\vdots & \ddots & \ddots & 0 & -\frac{\beta_2}{\sigma_{n-1}} \\[1mm]
0 & \cdots & 0 & \sigma_{n-1} & \beta_1}.
\]
Note that $A$ is of the form $A = D \, C \, D^{-1}$, where $C$ is a standard companion matrix, and
\[
D = {\rm diag}(1, \ \sigma_1, \ \sigma_1 \, \sigma_2, \ \dotsc, \ \sigma_1 \cdots \sigma_{n-1}),
\]
where the scaling preserves the spectrum.
It is easy to see that Arnoldi started from $e_1$ again generates the canonical basis.
Hence, for each $1\le k\le n-1$, the matrix $H_{k+1,k}$ has subdiagonal entries $\sigma_1,\dots,\sigma_k$, so
\[
\|H_{k+1,k}\| = \sigma_k.
\]
Furthermore, for the eigenvalues, it is not hard to verify that
\[
\det(A-\lambda I)
=
(-1)^n \, (\lambda^n - \beta_1 \, \lambda^{n-1} + \beta_2 \, \lambda^{n-2} - \cdots + (-1)^n \, \beta_n) = (-1)^n \, \pi(\lambda).
\]
Thus the roots of $\det(A-\lambda I)$ are exactly $\lambda_1,\dots,\lambda_n$.

For a general unit vector $b$, choose a unitary matrix $Q$ with first column $b$ and replace $A$ by $Q\,A\,Q^\top$. The Arnoldi process is unchanged up to unitary equivalence, so the same Hessenberg matrices are produced, and unitary similarity preserves the eigenvalues.
\end{proof}

While the prior results exhibit arbitrarily bad convergence of a TF Krylov method, we investigate the typical behavior in the next two examples.

\begin{example} \rm
In \Cref{fig:1}, we show  behavior of Krylov spaces using the transpose (blue curves) versus TF Krylov methods (red curves) for approximating the largest and smallest singular values.
The matrix on the left hand side of ~\Cref{fig:1} is $A = Q_1 \cdot {\rm diag}(1, \dotsc, 1000) \cdot Q_2$, where $Q_1$ and $Q_2$ are random orthogonal matrices, so that $\|A\| = \sigma_{\max}(A) = 1000$ and $\|A^{-1}\|^{-1} = \sigma_{\min}(A) = 1$.
The matrix on the right hand side is constructed in the same way, but $\text{diag}(1, \dots, 100)$ is replaced with $\text{diag}(\sigma_i)$, where
$\sigma_i = 1 + \frac{999}{i}$,
so that the $n$ singular values decay algebraically with $\sigma_{\max} = 1000$ and $\sigma_{\min} \approx 2$.
Note also that this results in a large gap between $\sigma_1 = \sigma_{\max}$ and $\sigma_2$. In all plots, the dotted line represents the true value of $\sigma_{\min}$ or $\sigma_{\max}$, and the blue curves indicate the average convergence curves of approximations to the extreme singular values, using 10 different initial random vectors $v_1$.
For every $v_1$, we run Lanczos bidiagonalization to obtain the matrix relations
\[
A \, V_k = U_k \, B_{k,k}, \qquad A^\top U_k = V_{k+1} \, (B_{k,k+1})^\top,
\]
where $B_{k,k+1}$ is a $k \times (k+1)$ upper bidiagonal matrix.
For $j = 1, \dotsc, 25$, we then plot $\sigma_{\max}(B_{j+1,j})$ and $\sigma_{\min}(B_{j+1,j})$, after $2j$ matvecs ($j$ with $A$ and $j$ with $A^\top$).
This means that we approximate from the Krylov spaces $\calk_j(A^\top \! A, \, v_1)$.
As is well known, the convergence to $\sigma_{\max}$ is usually very fast, while the approximation of the $\sigma_{\min}$ is much harder. 

We observe how hard it may be to approximate extreme singular values without the transpose. The red curves represent the average convergence using the TF Krylov spaces $\calk_j(A, \, v_1)$, for $j = 1, \dotsc, 50$ (using an equal number of matvecs for a fair comparison). In the first case of equispaced singular values between 1 and 1000 (left side of~\Cref{fig:1}), the TF method does much worse than the transpose method in approximating both $\sigma_{\max} = 1000$ and $\sigma_{\min} = 1$.
 In the case of  algebraic singular value decay (right side of~\Cref{fig:1}),  the TF method also struggles to approximate $\sigma_{\max} = 1000$, but performs  better than the transpose method in approximating $\sigma_{\min} \approx 2$.

\begin{figure}[htbp]
\centering

\begin{subfigure}{0.48\textwidth}
    \centering
    \begin{overpic}[width=0.8\textwidth]{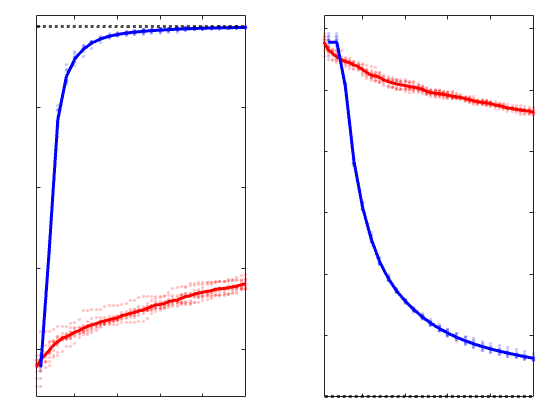}

\put(20,75){\small $\sigma_{\max}$}
\put(70,75){\small $\sigma_{\min}$}

\put(19,-5){\footnotesize  matvecs}
\put(69,-5){\footnotesize matvecs}

\put(15,60){\color{blue} transpose}
\put(27,25){\color{red} TF}

\put(65,64){\color{red} TF}
\put(68,30){\color{blue} transpose}

\put(11,0){\scriptsize{10}}
\put(18.5,0){\scriptsize{20}}
\put(26.5,0){\scriptsize{30}}
\put(34,0){\scriptsize{40}}
\put(41.5,0){\scriptsize{50}}

\put(62.5,0){\scriptsize{10}}
\put(70,0){\scriptsize{20}}
\put(77.5,0){\scriptsize{30}}
\put(85.5,0){\scriptsize{40}}
\put(93,0){\scriptsize{50}}

\put(-2,11.5){\scriptsize 600}
\put(-2,25.5){\scriptsize 700}
\put(-2,40){\scriptsize 800}
\put(-2,54.5){\scriptsize 900}
\put(-4,69){\scriptsize 1000}

\put(54,3){\scriptsize 0}
\put(50,13.5){\scriptsize 100}
\put(50,24.5){\scriptsize 200}
\put(50,35.5){\scriptsize 300}
\put(50,46.5){\scriptsize 400}
\put(50,57.5){\scriptsize 500}
\put(50,68.5){\scriptsize 600}

\end{overpic}
\end{subfigure}
\hfill
    \begin{subfigure}{0.48\textwidth}
    \centering
    \begin{overpic}[width=0.8\textwidth]{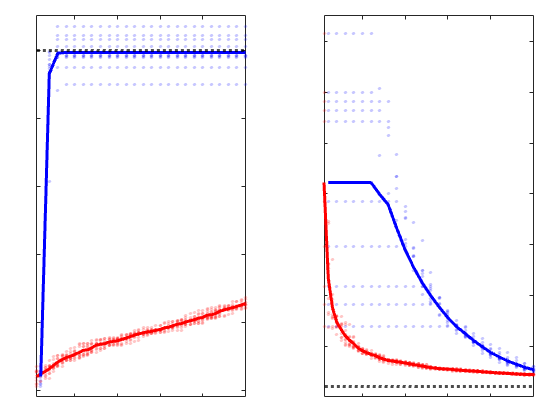}

\put(20,75){\small $\sigma_{\max}$}
\put(70,75){\small $\sigma_{\min}$}

\put(19,-5){\footnotesize  matvecs}
\put(69,-5){\footnotesize matvecs}

\put(15,55){\color{blue} transpose}
\put(27,21){\color{red} TF}

\put(68,11.5){\color{red} TF}
\put(60,45){\color{blue} transpose}

\put(11,0){\scriptsize{10}}
\put(18.5,0){\scriptsize{20}}
\put(26.5,0){\scriptsize{30}}
\put(34,0){\scriptsize{40}}
\put(41.5,0){\scriptsize{50}}

\put(62.5,0){\scriptsize{10}}
\put(70,0){\scriptsize{20}}
\put(77.5,0){\scriptsize{30}}
\put(85.5,0){\scriptsize{40}}
\put(93,0){\scriptsize{50}}

\put(3,4){\scriptsize 0}
\put(-1.5,16){\scriptsize 200}
\put(-1.5,28){\scriptsize 400}
\put(-1.5,40.5){\scriptsize 600}
\put(-1.5,52.5){\scriptsize 800}
\put(-3.8,64.5){\scriptsize 1000}

\put(54,3){\scriptsize 0}
\put(52.5,12){\scriptsize 10}
\put(52.5,21){\scriptsize 20}
\put(52.5,30){\scriptsize 30}
\put(52.5,39){\scriptsize 40}
\put(52.5,48){\scriptsize 50}
\put(52.5,57){\scriptsize 60}
\put(52.5,66){\scriptsize 70}

\end{overpic}
    \label{fig:1b}
\end{subfigure}

\caption{A comparison of the approximation of the largest and smallest singular values versus the number of matrix-vector products for two  matrices. Left: random $1000 \times 1000$ matrix with singular values $1, 2, \dots, 1000$. Right:   $1000 \times 1000$ matrix with singular values algebraically decaying  from 1000 to about 2. The curves plot the averages over 10 different initial random vectors, and the faint blue and red scattered points represent the 10 different trials. The black dotted horizontal lines show the true values of $\sigma_{\max}$ and $\sigma_{\min}$.}
\label{fig:1}
\end{figure}
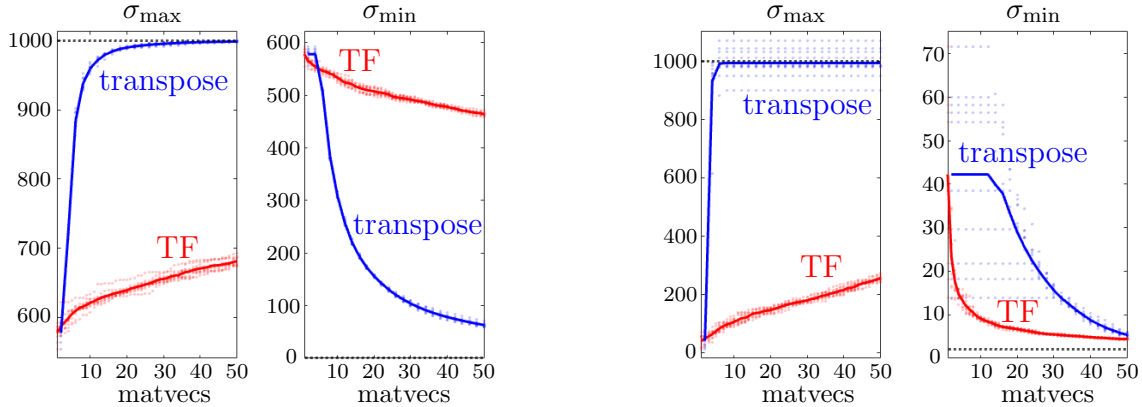
\end{example}
We now consider a family of finite difference matrices arising from PDEs. 
\begin{example} \rm We repeat the previous experiment for two $1000 \times 1000$ finite difference matrices for a convection-diffusion operator with a reaction term, i.e.,   
\[
\mathcal{L}u = u_{xx} + \alpha \, u_x + u \quad \text{on} \quad (0, 1), \qquad u(0) = u(1) = 0,
\]
where $\alpha$ parametrizes the convection term. Setting $\alpha=0$ gives a symmetric discretization. Increasing $\alpha$ makes the problem increasingly convection-dominated and the discretized operator increasingly nonnormal.

In \Cref{fig:2}, we display the plots for the values $\alpha = 1$ and $\alpha = 10^4$. In both cases,  the TF approximation of the extreme singular values performs better than the method using the transpose, and both methods do well in approximating $\sigma_{\max}$ and struggle to approximate $\sigma_{\min}$. In approximating $\sigma_{\min}$, the TF method actually outdoes the transpose method for both values of $\alpha$. In the case of $\alpha = 10^4$, the TF method and transpose method do about equally poorly in approximating $\sigma_\text{min}$. 
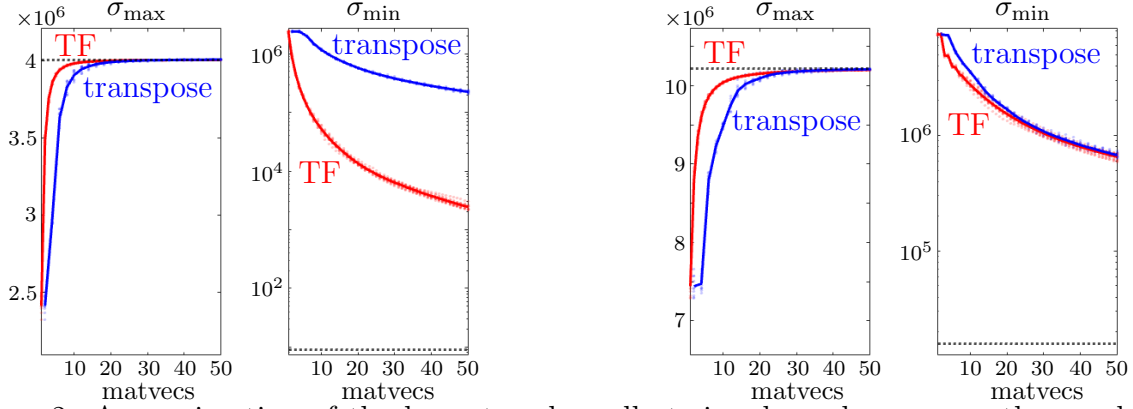
\begin{figure}[t]
\centering
    \begin{subfigure}{0.48\textwidth}
    \centering
    \begin{overpic}[width=0.8\textwidth]{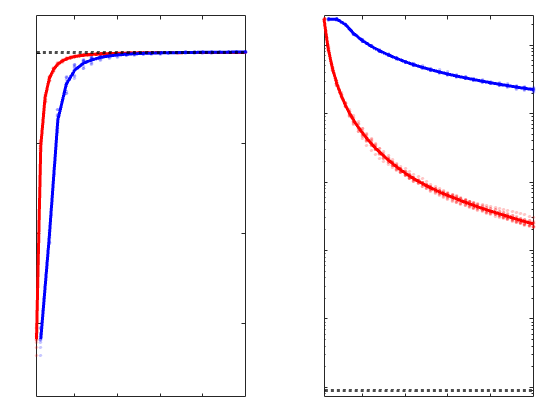}

\put(20,75){\small $\sigma_{\max}$}
\put(70,75){\small $\sigma_{\min}$}

\put(19,-5){\footnotesize  matvecs}
\put(69,-5){\footnotesize matvecs}

\put(15,58){\color{blue} transpose}
\put(9,66.5){\color{red} TF}

\put(60,40){\color{red} TF}
\put(67,67){\color{blue} transpose}

\put(11,0){\scriptsize{10}}
\put(18.5,0){\scriptsize{20}}
\put(26.5,0){\scriptsize{30}}
\put(34,0){\scriptsize{40}}
\put(41.5,0){\scriptsize{50}}

\put(62.5,0){\scriptsize{10}}
\put(70,0){\scriptsize{20}}
\put(77.5,0){\scriptsize{30}}
\put(85.5,0){\scriptsize{40}}
\put(93,0){\scriptsize{50}}

\put(-.5,15.5){\scriptsize $2.5$}
\put(3,32){\scriptsize $3$}
\put(-.5,48){\scriptsize $3.5$}
\put(3,64){\scriptsize $4$}
\put(1, 73){\scriptsize $\times 10^6$}

\put(50,16.5){\scriptsize $10^2$}
\put(50,40.5){\scriptsize $10^4$}
\put(50,65){\scriptsize $10^6$}

\end{overpic}
\end{subfigure}
\hfill
\begin{subfigure}{0.48\textwidth}
    \centering
        \begin{overpic}[width=0.8\textwidth]{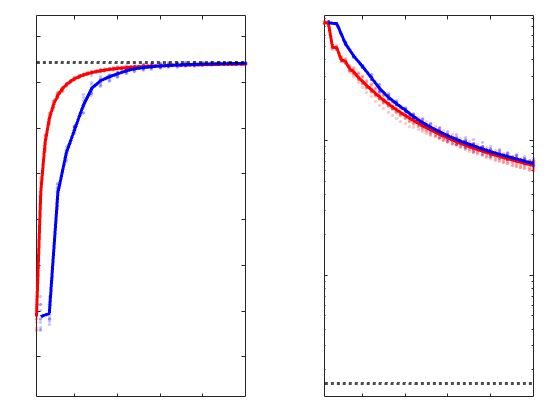}

\put(20,75){\small $\sigma_{\max}$}
\put(70,75){\small $\sigma_{\min}$}

\put(19,-5){\footnotesize  matvecs}
\put(69,-5){\footnotesize matvecs}

\put(15,51){\color{blue} transpose}
\put(9,65){\color{red} TF}

\put(60,50){\color{red} TF}
\put(65,65){\color{blue} transpose}

\put(11,0){\scriptsize{10}}
\put(18.5,0){\scriptsize{20}}
\put(26.5,0){\scriptsize{30}}
\put(34,0){\scriptsize{40}}
\put(41.5,0){\scriptsize{50}}

\put(62.5,0){\scriptsize{10}}
\put(70,0){\scriptsize{20}}
\put(77.5,0){\scriptsize{30}}
\put(85.5,0){\scriptsize{40}}
\put(93,0){\scriptsize{50}}

\put(3,10){\scriptsize $7$}
\put(3, 26){\scriptsize $8$}
\put(3,42.5){\scriptsize $9$}
\put(1,58.5){\scriptsize $10 $}
\put(1, 73){\scriptsize $\times 10^6$}

\put(50,24){\scriptsize $10^5$}
\put(50,48.5){\scriptsize $10^6$}

\end{overpic}
\end{subfigure}

\caption{Approximation of the largest and smallest singular values versus the number of matrix-vector products for finite-difference discretizations of $u_{xx}+\alpha u_x+ u$ with two different values of $\alpha$. Left: $\alpha = 1$. Right: $\alpha = 10^4$. The right panels, corresponding to the smallest singular value, use a logarithmic scale on the $y$-axis. The curves plot the averages over 10 different initial random vectors, and the faint blue and red scattered points represent the 10 different trials. The black dotted lines show the true values of $\sigma_{\max}$ and $\sigma_{\min}$.}
\label{fig:2}
\end{figure}
\end{example}

\subsection{Non-identifiability of a matrix norm}\label{sec:spectral_norm_identifiability} 
To estimate any matrix norm using matrix-vector products, we establish a lower bound of $q \geq n$ matrix-vector products to ensure there is a unique solution to the problem. The gist of the result is that if one only has access to $n-1$ matrix-vector products with the matrix $A$, one can ``hide'' an arbitrarily large singular vector in a direction orthogonal to the $n-1$ inputs. As a result, one can construct a matrix $B$ satisfying the same $n-1$ matrix-vector products as $A$, but with an arbitrarily different norm. In particular, this result applies to the spectral, Frobenius, and Schatten norms. 

\begin{proposition}\label{prop:spectral_norm_identifiability}
Let $A \in \mathbb{R}^{m\times n}$, $X\in \mathbb{R}^{n\times q}$ with $q<n$, and $\|\cdot\|_*$ be any matrix norm. Then, for every $\delta>0$, there exists
$B\in\mathbb{R}^{m\times n}$ such that $BX=AX$ but
\[
\|B\|_* \ge \delta .
\]
\end{proposition}

\begin{proof}
Since $X^\top\in\mathbb{R}^{q\times n}$, its null space is nontrivial. Choose $u\in\mathbf{R}^m$ and $v\in \operatorname{null}(X^\top)$ with $\|v\|=\|u\|=1$. For $\tau\ge 0$,
define
\[
B=A+\tau \, uv^\top .
\]
Then $v^\top X=0$, and hence $BX=AX+\tau \, uv^\top X=AX$. Moreover, by the reverse triangle inequality,
\[
\|B\|_* \ge \|\tau \, uv^\top\|_*-\|A\|_*
= \tau \, \|uv^\top\|_*-\|A\|_*.
\]
Since $uv^\top\neq 0$ and $\|\cdot\|_*$ is a matrix norm, $\|uv^\top\|_*>0$.
Therefore, choosing
\[
\tau \ge \|uv^\top\|^{-1} \, (\delta+\|A\|)
\]
gives $\|B\|_* \ge \delta$.
\end{proof}

This result highlights an important qualitative difference between query complexity lower bounds and non-identifiability results. While identifiability of the Frobenius norm requires $n$ queries, one can use $\mathcal O(\varepsilon^{-2})$  randomized, one-sided queries $\|Ag_i\|^2$, where $g_i$'s entries are distributed as i.i.d.  $\mathcal N(0, 1)$, and apply Hutchinson's estimator~\cite{hutchinson_89} to $A^\top \! A$ to approximate  $\|A\|_F$ within relative error $\eps$. Because the identifiability result relies on an adversarial construction that depends on the transcript, a randomized algorithm can avoid this adversary with high probability.

\section{Low-rank approximation}\label{sec:low-rank}
We now turn to low-rank approximation, a problem that is well-studied in the matrix-free context~\cite{bakshi2022low, bakshi2023krylov, halko2011finding, nakatsukasa2020fast,tropp2023randomized}. In particular, lower bounds on the two-sided query complexity of low-rank approximation have been thoroughly investigated~\cite{bakshi2022low, bakshi2023krylov, meyer2024unreasonable}. In this section, we derive a simple lower bound on the query complexity of this task in the TF setting. 

We see that low-rank approximation is one of the clearest examples in which TF access remains informative, but only at a prohibitive query cost. The key observation is that on exactly rank-$k$ inputs, any $(1+\eps)$-approximation must recover the matrix exactly. This allows us to reduce approximate low-rank approximation to exact reconstruction on a carefully chosen hard distribution. We note that a non-identifiability result for this problem where the matrix is exactly rank-$k$ is proven in~\cite{halikias2024structured}. Additionally, in the exactly rank-$k$ case, an analysis of the space of possible row spaces or matrices that can be recovered from only forward matvecs is given in~\cite{otto2023model} and~\cite{boulle2024operator}, respectively. 

\subsection{Query complexity lower bound} We prove a query complexity lower bound for the more general problem of finding a low-rank approximation that is within a constant factor $\gamma > 0$ of optimal. Because this problem is strictly easier than relative approximation within a prescribed error $\eps > 0$, the lower bound also extends to relative approximation. 

\begin{theorem}\label{thm:lra-lower-bound}
Let $A\in \R^{m \times n}$ with $n \ge m$ and $1 \le k \le m $. Let $\gamma > 1$ be an approximation factor, possibly depending on $k$, $n$, and $m$, and let $A_k$ denote the best rank-$k$ approximation to $A$. Suppose a randomized, adaptive algorithm with access only to the oracle $x \mapsto Ax$ outputs a rank at most $k$ matrix $\widetilde A$ satisfying
\[
\|A-\widetilde A\|_F \le \gamma \cdot \|A-A_k\|_F
\]
with probability at least $\tfrac23$ for every $A \in \R^{m \times n}$. Then, the algorithm must use at least $n$ forward matrix-vector products.
\end{theorem}
\begin{proof}
By Yao's minimax principle, it suffices to prove the claim for deterministic algorithms against a hard distribution. Consider the distribution
\[
A = e_1 \, g^\top, \qquad g \sim \caln(0,I_n).
\]
Every realization of this distribution has rank one, so for every $k \ge 1$,
\[
\|A-A_k\|_F = 0.
\]
Hence, any successful algorithm must output $\widetilde A = A$ exactly.
In other words, for this hard distribution, the approximation problem collapses to exact recovery of the unknown matrix.

Fix a deterministic adaptive algorithm that makes $q<n$ forward queries. Let $x_1,\dots,x_q \in \R^n$ be the queried vectors, let $X=[x_1,\dots,x_q]$, and let the transcript be
\[
AX = e_1 \, g^\top X.
\]
Conditioned on the  transcript, the vector $g$ is constrained only by the $q$ linear equations encoded by $g^\top X$. Since $q<n$, the conditional law of $g$ remains a nondegenerate Gaussian on an affine subspace of dimension at least $n-q>0$. In particular, conditioned on the transcript, the true matrix $A=e_1g^\top$ is not determined uniquely. There is still a continuum of matrices consistent with the same observed forward products.

Because the algorithm's output is measurable with respect to the transcript, the conditional probability that it outputs the exact matrix $A$ is zero. Therefore the deterministic algorithm succeeds with probability zero on this distribution whenever $q<n$. Then, by Yao's minimax principle, any randomized, possibly adaptive algorithm succeeding with probability at least $\tfrac23$ on every input must use at least $n$ forward queries.
\end{proof}
The theorem is sharp in a qualitative sense: in the TF model, near-optimal low-rank approximation is not a genuinely matrix-free task. On exact rank-$k$ inputs, approximate recovery collapses to exact recovery, which requires full-dimensional information.

\subsection{Low-rank approximation via column subset selection}

The problem of column subset selection is well-studied in numerical linear algebra, particularly as it relates to matrix factorizations like  interpolative/CUR decompositions~\cite{boutsidis2014optimal,goreinov1997theory, mahoney2009cur, voronin2017efficient}, adaptive cross approximation~\cite{bebendorf2011adaptive}, and the discrete empirical interpolation method (DEIM)~\cite{chaturantabut2010nonlinear}. We now discuss the problem of low-rank approximation via column subset selection, also known as finding a $CC^\dagger$ or $CX$ interpolative decomposition. The goal is to find an optimal or near-optimal set of $k$ columns of  $A \in \R^{m \times n}$, stored as the columns of $C \in \R^{m \times k}$ such that the low-rank approximation given by the projection $CC^\dagger A$ satisfies
\begin{equation}\label{eq:css_lra}
\|A - CC^\dagger A \|_F \le \gamma_k \cdot \|A - A_k \|_F. 
\end{equation} 

Because one is restricted to using $A$'s columns for low-rank approximation, the theoretical best possible constant is $\gamma_k \geq \sqrt{k+1}$~\cite{deshpande2006matrix}. 

In the matrix-free setting of column subset selection,  one may sample columns or rows by querying elementary basis vectors. Recently, both deterministic and randomized algorithms using matrix-vector products have been developed for matrix-free column subset selection with the guarantee~\eqref{eq:css_lra}~\cite{cortinovis2026adaptive,osinsky2025close}.
However, one cannot form $CC^\dagger A$ efficiently without transpose access, as one must apply $C^\dagger$ to every column of $A$.  Hence, a key assumption in the randomized matvec algorithm in~\cite{cortinovis2026adaptive} is that in addition to forward matvec access with $A$, one also has a good approximation to $A$'s row space, i.e., a matrix $V \in \R^{n \times k}$ such that $A\approx A VV^\top$. Prior results on TF row-space approximation already show that one cannot achieve this guarantee using only forward queries unless one performs $n$ matvecs~\cite{boulle2024operator,otto2023model}. 

If the task is just to find a good set of $k$ columns for low-rank approximation \textit{without} explicitly forming the low-rank approximation $CC^\dagger  A$, we can prove a non-identifiability result by reducing to the case where $A$ is exactly rank-$k$:
\begin{proposition}\label{prop:css_identifiability}
    Let $A \in \R^{m \times n}$ satisfy $\text{rank}(A) = k \leq \min\{m, n \}$, and let $X \in \R^{n \times (n-1)}$ with full column rank. Let $S \subset \{1, \dots, n \}$, $|S| = k$ represent a subset  of $k$ columns such that $\text{col}( A_{:, S}) = \text{col}(A)$. Suppose that there exists $v \in \R^n$ satisfying $v^\top X = 0$ for which there exists $i \in S$ such that $v_i \neq 0$ and $v \not \in \text{row}(A)$.  Then, there exists a matrix $B \in \R^{m \times n}, \text{rank}(B) = k$, such that $BX = AX$, but $\text{col}( B_{:, S}) \neq \text{col}(B)$.  
\end{proposition}
\begin{proof}
Write the economized QR factorization of $A = QR$, where $Q \in \R^{m \times k}$ has orthonormal columns, and $R \in \R^{k \times n}$. Because 
\[
    A_{:, j} = Q R_{:, j} \iff A_{:, S} = Q R_{:, S},
\]
we have that $\text{col}(A_{:, S}) = \text{col}(A)$ if and only if $R_{:, S} \in \R^{k \times k} $ is invertible. We  construct a rank-$k$ matrix $B\in\R^{m \times n} $ satisfying $BX = AX$, but $B = Q\wt R$, where $\wt R_{:, S}$ is singular. 

Because $v_S \neq 0$ and $R_{:, S}$ is invertible, we can find a vector $u \in \R^{k}$ such that $v_S^\top R_{:, S}^{-1} u = -1$. Then, define $\wt R$  as the rank-1 update $R + uv^\top$, so that
\[
B = Q \wt R = Q \, (R + uv^\top) = A + Q\,uv^\top .
\]
Moreover,  $\text{rank}(B) = k$ because $v \not \in \text{row}(A)$. To see why this is true, we prove the contrapositive. If $\text{rank}(B) = k-1$, then $\text{rank}(R + uv^\top) = k-1$, so $R + uv^\top$ has a nontrivial left nullspace, i.e., there exists a vector $w \in \R^k$ such that $w^\top (R + uv^\top) = w^\top R + (w^\top u) v^\top = 0$. It must be the case that $w^\top u\neq 0$, otherwise $w^\top R = 0$, which contradicts the fact that $R$  has full row-rank. Then, solving for $v$ shows that it is contained in $\text{row}(R) = \text{row}(A)$.
 
By the formula for the determinant of a rank-one update,
\[
\det(\wt R_{:, S} ) = \det(R_{:, S}) (1 + v_S^\top R_{:, S}^{-1} \, u) = 0.
\]
Hence, $\wt R_{:, S}$ is singular, so $\text{col}(B_{:, S}) \neq \text{col}(B)$.
\end{proof}

We explain in the following remark that the assumptions about $X$ in~\Cref{prop:css_identifiability} are necessary for non-identifiability to hold.

\begin{remark} \rm
In~\Cref{prop:css_identifiability}, the conditions that $v_S \neq 0$ and $v \not \in \text{row}(A)$ are necessary.  If for all $v \in \R^n$ satisfying $v^\top X = 0$, $v_S = 0$, then  $\text{span} \{ e_j : j \in S \} \subset \text{col}(X)$, where $e_j$ is the $j$th elementary basis vector. Then, for each $j \in S$, there exists a $y_j$ such that $e_j = X y_j$, and so $A_{:, j} = Ae_j =  (AX)y_j$ is determined by the transcript. Thus, the selected columns $A_{:, S}$ are identifiable from the transcript.  In this case, one got lucky and queried the  columns in $S$ up to linear combination,  only needing $k$  queries to identify  $S$.
   
Moreover, we require that $v \not \in \text{row}(A)$ to ensure that $B$ is still rank-$k$. If $v \in \text{row}(A)$, consider the case where $k = 1$ and $q =n-1$, so any rank-1 matrix $wz^\top$ satisfies $wz^\top X = 0$, where $w \in \R^m$ and $z \in \R^n$.  Since $X \in \R^{n \times (n-1)}$, the direction of $z$ is uniquely determined, i.e., $z = \alpha a$ for some $a \in \R^n$, $\alpha \neq 0$. Because $wz^\top$ is rank-1, any nonzero column of $wz^\top$ spans $\text{col}(wz^\top)$, and this is given by any $s \in \{1, \dots, n\}$ for which $a_s \neq 0$. Thus, $s$ is  identifiable from the query $uv^\top X = 0$.
\end{remark}

We have shown that, excluding two edge cases, low-rank approximation via column subset selection also fundamentally relies on transpose information, in the sense that the set of columns spanning the range of the matrix is not identifiable from fewer than $n$ forward queries. In the following section, we address the related, more challenging problem of finding a maximum volume submatrix.  

\section{Local maximum volume submatrix}\label{sec:css}

In this section, we consider identifiability for finding the local maximum volume submatrix in a TF manner. To describe the maxvol problem, for $S\subset\{1,\dots,n\}$ with $|S|=k$, write
\[
A_S = A_{:,S},
\qquad
\operatorname{vol}(A_S)
=
\textstyle \prod_{\ell=1}^k \sigma_\ell(A_S)
=
\det(A_S^\top A_S)^{1/2}.
\]
The global maxvol problem is to find
\[
\max_{S \subset \{1, \dotsc, n\}, \, |S| = k} {\rm vol}(A_S).
\]
This problem is closely related to the problem of low-rank approximation via  column subset selection described in the previous section. There are many contexts where one wishes to find a representative or dominant submatrix~\cite{goreinov2010find} and a variety of techniques using volume sampling/determinantal point processes~\cite{deshpande2010efficient} and leverage scores.  While the global maxvol problem is NP-complete~\cite{shitov2021column}, it has been recently shown that the problem of finding the near-local maxvol of a matrix $A$ is theoretically necessary for Gaussian elimination or the QR algorithm to be rank-revealing~\cite{damle2025estimating}. Moreover, this work proposes a simple greedy algorithm for this problem. The local maxvol solution also results in a choice of column indices that yield a near-optimal low-rank approximation, thus solving the column subset selection problem. However, an important distinction is that while column subset selection can achieve good low-rank approximation error in terms of the trailing singular values (as proven for, e.g.,  the methods of~\cite{osinsky2025close} and~\cite{cortinovis2026adaptive}), it does not necessarily guarantee a near-local maximum volume pivot or good singular value estimates. 

The local maxvol problem is defined as follows. For a size-$k$ set $S$, define its one-swap neighborhood by
\[
\caln_1(S)
=
\bigl\{(S\,\setminus\{i\})\cup\{j\}: i\in S,\ j\notin S\bigr\}.
\]
For $\gamma\ge 1$, we say that $S$ is a $\gamma$-local maximum-volume
subset of $A$ if
\begin{equation}\label{eq:gamma-local-maxvol}
\operatorname{vol}(A_T)
\le
\gamma \cdot \operatorname{vol}(A_S)
\qquad
\text{for every } T\in \caln_1(S).
\end{equation}
The case $\gamma=1$ is ordinary one-swap local maxvol.

\subsection{Identifiability of local maximum volume}

We now consider identifiability for the local maxvol problem.
\begin{proposition}
\label{prop:css-near-local-maxvol-nonidentifiability}
Let $A\in\R^{m\times n}$ satisfy $\operatorname{rank}(A)>k$, and let
$X\in\R^{n\times q}$ have full column rank with $q\le n-1$. Let
$\wh S\subset\{1,\dots,n\}$, $|\wh S|=k$, satisfy
$
\operatorname{vol}(A_{\wh S})>0.
$
Then, for every $\gamma\ge 1$, there exists a matrix $B\ne A$ such that
$
BX=AX,
$
but $\wh S$ is not a $\gamma$-local maximum-volume subset of $B$.
More precisely, there exists a one-swap neighbor
$T\in\caln_1(\wh S)$ such that
\[
\operatorname{vol}(B_T)
>
\gamma \cdot \operatorname{vol}(B_{\wh S}).
\]
\end{proposition}

\begin{proof}
Let
$
N=\operatorname{null}(X^\top).
$
Since $q\le n-1$, the subspace $N$ is nontrivial. We consider two cases.

{\bf Case 1:} First suppose there exists $v\in N$ whose restriction to $\wh S$ is not
identically zero. Pick $i\in \wh S$ such that $\alpha := v_i\ne 0$. Since
$\operatorname{vol}(A_{\wh S})>0$, the columns of $A_{\wh S}$ are
linearly independent. Since $\operatorname{rank}(A)>k$, there exists
$j\notin \wh S$ such that
$
a_j\notin \operatorname{col}(A_{\wh S}),
$
where $a_j$ denotes the $j$th column of $A$.

Define 
\[
u= - \alpha^{-1} \, a_i,
\qquad
B=A+u \, v^\top.
\]
Because $v\in N$, we have $v^\top X=0$, and hence $BX=AX+u \, v^\top X=AX$.
Moreover, $b_i=a_i+\alpha \, u=a_i-a_i=0$, so $B_{\wh S}$ has a zero column.
Therefore
\begin{equation}\label{eq:collapsed-selected-volume}
\operatorname{vol}(B_{\wh S})=0.
\end{equation}
Now define the one-swap neighbor
$T=(\wh S\,\setminus\{i\})\cup\{j\}$.
We claim that $B_T$ has full column rank. Suppose that $\alpha_\ell, \beta \in \R$ satisfy
\[
\textstyle \sum_{\ell\in \wh S\,\setminus\{i\}} \alpha_\ell \, b_\ell
+ \beta \, b_j = 0.
\]
Since
$b_\ell = a_\ell-\frac{v_\ell}{v_i} \, a_i$ for every $\ell$, the preceding relation becomes
\[
\textstyle \sum_{\ell\in \wh S\,\setminus\{i\}}\alpha_\ell \, a_\ell
+ \beta \, a_j - v_i^{-1} \,
\big(
\sum_{\ell\in \wh S\,\setminus\{i\}} \alpha_\ell \, v_\ell
+ \beta \, v_j \big) \, a_i = 0.
\]
All terms except $\beta \, a_j$ lie in $\operatorname{col}(A_{\wh S})$.
Since $a_j\notin \operatorname{col}(A_{\wh S})$, we must have
$\beta=0$. The remaining relation is then a linear dependence among the
columns of $A_{\wh S}$, which are linearly independent. Hence all
$\alpha_\ell=0$. Thus $B_T$ has linearly independent columns and hence,
$
\operatorname{vol}(B_T)>0.
$
Combining this with~\eqref{eq:collapsed-selected-volume}, we get
\[
\operatorname{vol}(B_T)
>
\gamma \cdot \operatorname{vol}(B_{\wh S})
\]
for every $\gamma\ge 1$. Hence $\wh S$ is not a
$\gamma$-local maximum-volume subset of $B$.

{\bf Case 2:} Every $v\in N$ vanishes on
$\wh S$. In this case, the columns indexed by $\wh S$ are already
determined by the transcript $AX$. Indeed, for each $i\in \wh S$, the
condition $v_i=0$ for every $v\in N$ implies
\[
e_i\in N^\perp=\operatorname{col}(X).
\]
Therefore, if $B X=A X$, then
$
b_i-a_i=(B-A) \, e_i=0$ for every $i\in \wh S$. 
Thus one cannot, in general, force $B_{\wh S}$ to have zero volume in
this case. Instead, we must make a neighboring subset of columns have a large volume. To this end, choose a nonzero $v\in N$. Since every vector in $N$
vanishes on $\wh S$, but $v\ne 0$, there exists $j\notin \wh S$ such
that $v_j\ne 0$. Pick any $i\in \wh S$. Because $\operatorname{vol}(A_{\wh S})>0$, the matrix $A_{\wh S\,\setminus\{i\}}$ has full column
rank. Choose a unit vector
$
u\perp \operatorname{col}(A_{\wh S\,\setminus\{i\}}),
$
and, for $\tau \in\R$, define
$
B_{\tau} = A+\tau \, u \, v^\top.
$
Again, since $v^\top X=0$, we have
$
B_{\tau} X=AX.
$
Moreover, $v$ vanishes on $\wh S$, so the selected columns are unchanged:
$
(B_{\tau})_{\wh S}=A_{\wh S}.
$
Hence
\begin{equation}\label{eq:selected-volume-fixed}
\operatorname{vol}\bigl((B_{\tau})_{\wh S}\bigr)
=
\operatorname{vol}(A_{\wh S})
>
0.
\end{equation}
Now consider the one-swap neighbor
$
T=(\wh S\,\setminus\{i\})\cup\{j\}.
$
The columns of $(B_{\tau})_T$ consist of the columns of $A_{\wh S\,\setminus\{i\}}$ together with
$
a_j+\tau \, v_j \, u.
$

By the Schur complement formula,
\[
\operatorname{vol}((B_{\tau})_T)
=
\operatorname{vol}(A_{\wh S\,\setminus\{i\}}) \cdot
\|P_{\operatorname{col}(A_{\wh S\,\setminus\{i\}})^\perp} (a_j+\tau \, v_j \, u)\|.
\]
where we use the convention that $\det((A_{\wh S\,\setminus\{i\}})^\top A_{\wh S\,\setminus\{i\}})=1$ when $|\wh S|=1$. Since
$u\in \operatorname{col}(A_{\wh S\,\setminus\{i\}})^\perp$  and $v_j\ne 0$, the final factor grows as $|\tau \, v_j|$  as $|\tau|\to\infty$.
Therefore
$\operatorname{vol}\bigl((B_{\tau})_T\bigr)\to\infty$
as $|\tau|\to\infty$. Using~\eqref{eq:selected-volume-fixed}, we can choose $|\tau|$ sufficiently
large that
\[
\operatorname{vol}\bigl((B_{\tau})_T\bigr)
>
\gamma \cdot \operatorname{vol}(A_{\wh S})
=
\gamma \cdot \operatorname{vol}\bigl((B_{\tau})_{\wh S}\bigr).
\]
Taking $B=B_{\tau}$ gives the desired matrix in the second case.
\end{proof}

This proposition is a pure identifiability result: the issue is not computational efficiency, but the fact that the forward transcript does not determine the correct subset. Even exact optimization is impossible when distinct matrices agree on all queried products but induce different optimal column choices.

\section{Schatten-$p$ norm estimation}\label{sec:schatten}
We next consider Schatten-norm estimation. For $p \ge 1$, let
\[
\|A\|_{S_p} = \big( \textstyle \sum_i \sigma_i^p(A) \big)^{1/p}
\]
denote the Schatten-$p$ norm. For even $p$, we have $\|A\|_{S_{p}}^{p} = \operatorname{tr}((A^\top \! A)^{p/2})$. Important examples of this quantity include the nuclear norm (Schatten-1), Frobenius norm (Schatten-2), and spectral norm ($p \to \infty$). 

Schatten-$2p$ norm estimation from samples is well-studied; in the context of covariance matrices, it is used via the method of moments to approximate the eigenvalues from observed samples drawn from the covariance matrix's corresponding distribution~\cite{kongvaliantspectrum2017}. This algorithm produces an unbiased estimator for the Schatten-$2p$ norm using only one-sided matrix-vector products, and the variance of this estimator can be further reduced under additional assumptions, such as rapid decay of singular values~\cite{chucortinovisimproved2025}. In the case of Frobenius norm estimation ($2p=2$), this algorithm reduces to Hutchinson's estimator applied to $A^\top A$. 

The query complexity of Schatten-norm and spectral norm estimation is studied in~\cite{linguyenwoodruffon2014}. They consider the bilinear sketching model, where there is a distribution over $r \times n$ matrices $S$ and $n \times s$ matrices $T$ so that one observes $SAT$ to approximate $\|A\|$ within a constant factor. For even $p \geq 4$, they  obtain an $\varepsilon$-approximation to $\|A\|_{S^p}$ with $rs = \mathcal O(n^{2 - 4/p})$, which is also shown to be optimal in $n$ and $p$ dependence with a matching lower bound. Given this constraint on $rs$, if one wants to minimize the total number of queries $r + s$, the solution is  $r = s = \mathcal O(n^{1 - 2/p})$. We note that this result can be converted into a TF  bound as follows. A non-adaptive TF algorithm using $q$ forward matvecs in this form requires $S$ to be invertible, so $ r = n$ and $s = q$. Thus, substituting $rs = nq$, a TF algorithm using non-adaptive, randomized queries of the form $AT$ requires $
q = \Omega(n^{1 - 4/p})$ queries for even $p > 4$. It is worth noting that the number of forward queries needed for spectral norm estimation, which can be recovered from $\| A \|_{S^p}$ as $p \to \infty$, is $n$; see a related discussion in~\cite[Section 5]{martinsson2020randomized}. In this section, we obtain a more general lower bound for possibly adaptive queries. However, we lose the descriptive dependence on $p$.

We focus first on Frobenius norm estimation, where transpose access leads to the classical contrast between Hutchinson's estimator and Hutch++. Hutchinson's estimator uses only one-sided information, while Hutch++ exploits a low-rank correction that relies on both $A$ and $A^\top$. The following theorem shows that this difference is not merely algorithmic: the improved $\eps^{-1}$ dependence is impossible in the TF model.

\begin{theorem}\label{thm:frobenius-lower-bound}
There exists a constant $\eta >0$ such that for every $0<\eps\le \eta$ and all sufficiently large $n \gtrsim \eps^{-2}$, any randomized, adaptive algorithm that outputs $\widehat a$ satisfying
\[
|\,\widehat a-\|A\|_F\,| \le \eps \, \|A\|_F
\]
with probability at least $\tfrac23$ for every $A \in \R^{n \times n}$ must use $q = \Omega(\eps^{-2})$ queries.
\end{theorem}

\begin{proof}
By Yao's minimax principle, it suffices to analyze deterministic adaptive algorithms against a two-hypothesis distribution. 
Consider the two hypotheses 

\[
H_0:\quad A = \tfrac{1}{\sqrt n} \, e_1 \, g^\top,
\qquad
H_1:\quad A = \tfrac{1+4\eps}{\sqrt n} \, e_1 \, h^\top,
\qquad g, h\sim \caln(0,I_n) \text{ iid}.
\]
Fix a deterministic, adaptive algorithm making $q$ queries. Since each response is a scalar multiple of $e_1$, the algorithm loses no information by replacing each query direction by its component orthogonal to the previous ones and normalizing; thus we may assume the queried vectors $x_1,\dots,x_q$ are orthonormal.

Under both $H_0$ and $H_1$, the transcript consists of the $q$ scalars
\[
\tfrac{\sigma_i}{\sqrt n} \, g^\top x_1, \dots, \tfrac{\sigma_i}{\sqrt n} \, g^\top x_q,
\qquad \sigma_0=1, \quad \sigma_1=1+4\eps,
\]
which are i.i.d.~Gaussian random variables with law $\caln(0, \, n^{-1} \, \sigma_i^2)$. The same is true for $H_1$. Meanwhile, under both hypotheses, 
\[
\|A\|_F = \tfrac{\sigma_i}{\sqrt n} \, \|z\|, \quad z \sim \mathcal N(0, I_n).
\]
By concentration of the $\chi^2_n$ distribution, if $n \gtrsim \eps^{-2}$, then with probability at least $0.99$,
\[
(1- \tfrac14 \, \eps) \, \sigma_i \le \|A\|_F \le (1+ \tfrac14 \, \eps) \, \sigma_i
\]
simultaneously under both hypotheses.

Therefore, any estimator achieving relative error at most $\eps$ with probability at least $\tfrac23$ yields a test that distinguishes $H_0$ from $H_1$ with constant bias. On the other hand, the Kullback--Leibler divergence between the two transcript distributions is
\[
\begin{array}{l}
\mathrm{KL}\!\left(\caln(0, \, n^{-1} \, \sigma_0^2)^{\otimes q} \ \middle\| \ \caln(0, \, n^{-1} \, \sigma_1^2)^{\otimes q}\right) \\[2mm]
\phantom{MMMMMMMMM} = q \cdot \mathrm{KL}\!\left(\caln(0,1) \ \middle\| \ \caln(0,(1+4\eps)^2)\right)
= \calo(q\,\eps^2).
\end{array}
\]
Pinsker's inequality\footnote{We gratefully thank Chris and Cameron Musco for their help on this topic.}
therefore shows that if the test has constant success probability, then $q = \Omega(\eps^{-2})$. This proves the lower bound when $n \gtrsim \eps^{-2}$; if $n$ is not much larger than $\varepsilon^{-2}$, one can always recover $A$ using $q = n$ matvecs. Thus, up to constants, Hutchinson's estimator's query complexity is  optimal among all TF algorithms.
\end{proof}
This theorem separates two phenomena that are often conflated. Hutchinson's method is already optimal in the one-sided model, whereas the improved complexity of Hutch++ for estimating $\|A\|_F$ fundamentally relies on access to both $A$ and $A^\top$. In the following remark, we distinguish this result from theory surrounding improved variants of Hutch++, which also address the trade-off between matvecs with $A$ and $A^\top$. 

\begin{remark} \rm
Hutch++ relies on a combination of low-rank approximation (using matvecs with both $A$ and $A^\top$) and stochastic trace estimation (using matvecs only with $A$, as in Hutchinson's estimator).  In~\cite{persson2022improved}, the optimal split between the two phases of the algorithm is studied. In particular, if $A$ exhibits rapid singular value decay, it admits a more accurate low-rank approximation, and therefore requires fewer two-sided matvecs. However, this only saves a constant number of matvecs with $A^\top$, and there is not an asymptotic difference in the number of queries to $A$ and $A^\top$. 
\end{remark}

The following corollary extends~\Cref{thm:frobenius-lower-bound} to general Schatten norms. Intuitively, as $p \to \infty$, estimating $\|A\|_{S_p}$ becomes harder because one requires more information about $A$'s spectrum. However, the following lower bound is independent of $p$, and therefore  very loose as $p$ grows. 

\begin{corollary}\label{cor:schatten-rank-one}
Let $p \in [1,\infty]$,  $A \in \R^{n \times n}$, and $\varepsilon > 0$. For all sufficiently large $n \gtrsim \eps^{-2}$, any randomized, adaptive algorithm that estimates $\|A\|_{S_p}$ to relative accuracy $\eps$ with constant success probability on every $A$ must use $q = \Omega(\eps^{-2})$ forward matrix-vector products.

\end{corollary}

\begin{proof}
On rank-one matrices, all Schatten norms are equal to the unique nonzero singular value. Hence, the Frobenius norm lower bound applies verbatim.
\end{proof}
The corollary shows that, for lower bounds, the Frobenius norm already captures the difficulty of estimating every Schatten norm in the TF model. At least at the level of worst-case query complexity, the rank-one case is sufficient to transfer the obstruction to the whole Schatten family.

\section{Least squares problems}\label{sec:least-squares}
We conclude the main body with a concrete non-identifiability result for overdetermined least squares problems; we note that a preliminary version of this result was included in the first author's dissertation~\cite[Prop.~5.2.1]{halikias2025structured}. We consider the least squares problem $Qx = b$, where $b$ is known and $Q$ is only accessed via forward matrix-vector products. We note this setting arises because $A$ may be too large to store, but one can still store the right-hand side $b$. Even in the simplest orthonormal setting, where the solution is just $Q^\top b$, we show that forward matvecs alone do not determine the minimizer. This result complements recent lower bounds on the query complexity of approximately solving linear systems~\cite{derezinski2026matrix}.
Let $\kappa(B) = \sigma_{\max}(B) \, / \, \sigma_{\min}(B)$ be the condition number of $B$, where $\sigma_{\max}$ and $\sigma_{\min}$ are the largest and smallest singular values of $B$, respectively. 

\begin{proposition}\label{prop:ls-nonuniqueness}
Let $Q \in \R^{m \times n}$ have orthonormal columns, and let
\[
x = \argmin_{z\in \R^n}\|Qz-b\| = Q^\top b.
\]
Let $X \in \R^{n \times q}$ have orthonormal columns with $q<n$ and suppose that $x \not \in \text{col}(X)$. Then, there exists a matrix $B \in \R^{m \times n}$ such that $BX = QX$ and  $\kappa(B) =1$, but the least squares solution
\[
y = \argmin_{z\in \R^n}\|Bz-b\|
\]
satisfies $\|y - x\| = 2 \cdot \text{dist}(x, \text{col}(X))$. In particular, $y \neq x$. 
\end{proposition}
\begin{proof}
    Let 
\[
v = (I-XX^\top) \, x \ / \ \|(I-XX^\top) \, x\|,
\]
so that $v^\top X = 0$ and $v^\top x \ne 0$. Note that $v$ is well-defined because $x \not \in \text{col}(X)$. Then, define 
\[
B = Q \, (I - 2 \, vv^\top).
\]
By construction,  $BX = QX$ and $\kappa(B) = 1$. Since $y$ is the least squares solution to $Bz = b$, we find that 
\[
y = B^\top b = (I - 2 \, vv^\top) \, Q^\top b = (I - 2 \, vv^\top) \, x.
\]
Thus, $y-x = - 2 \, vv^\top x$ and
$
\|y - x\| = 2 \, |v^\top x| =  2 \cdot \text{dist}(x, \, \text{col}(X)).
$
\end{proof}
In the following remark, we explain that the assumption that $x \in \text{col}(X)$ in~\Cref{prop:ls-nonuniqueness} is necessary to for non-identifiability.
\begin{remark} \rm
In \Cref{prop:ls-nonuniqueness}, it is necessary that $x \not \in \text{col}(X)$. If $x \in \text{col}(X)$, $x$ actually \textit{is} identifiable from only $q$ forward matvecs in the case that $b \in \text{col}(Q)$. This corresponds to the situation where one got  lucky and queried $Q$ with vectors whose span contains $x$. Algorithmically, after performing $q$ matrix-vector products to obtain $QX$, one would observe that $b$ is a linear combination of the output vectors, i.e., $b=\sum_{i =1 }^q \alpha_i Qx_i = Qx \implies x = \sum_{i = 1}^q \alpha_i x_i$  by linearity.
\end{remark}

Thus, even in the benign orthonormal setting, forward access to the design matrix does not determine the least squares minimizer. The obstruction again comes from an unseen direction: one can perturb the problem without changing any queried products, while moving the solution and keeping the matrix well conditioned. The distance between the solution $x$ and the query space $\text{col}(X)$  determines the identifiability of the least squares solution.

\section{Conclusions}\label{sec:conclusions}
We have shown that one-sided access to a matrix is often substantially weaker than standard matrix-free access with both $A$ and $A^\top$ available. On the pathology side, Arnoldi projected norms can behave arbitrarily, so forward-only Krylov information does not by itself provide a trustworthy proxy for the spectral norm. On the complexity side, near-optimal low-rank approximation requires $n$ forward matvecs, while Frobenius norm estimation requires $\Omega(\eps^{-2})$ queries when $n$ is sufficiently large. On the identifiability side, solutions to the local maximum volume submatrix problem and orthonormal least squares are not determined by forward transcripts alone. Taken together, these results suggest that transpose access is not merely a technical convenience, but a structural source of information.

This does not necessarily mean that the TF setting is hopeless. Rather, it suggests that when TF algorithms succeed, there must be some additional structure in the matrix at hand. For example, while adjoint-free operator learning may be impossible or highly inefficient for general operators, solution operators of elliptic PDEs provide indirect information about the adjoint through forward queries~\cite{boulle2024operator}. In other cases, one may have a priori knowledge about the operator's row space~\cite{cortinovis2026adaptive}, or noisy interactions with the transpose operator in the ``unmatched transpose'' setting. Moreover, our experiments in~\Cref{fig:1} and ~\Cref{fig:2} demonstrate that for matrices exhibiting singular value decay or arising from PDEs, TF methods can even outperform ones which use the transpose. Thus, a natural continuation of the present work is to move beyond worst-case analyses and understand the performance of TF methods for  structured matrices that commonly arise in practice.

Many other natural directions remain open for exploration. Are there stronger lower bounds for Schatten-$p$ norm estimation beyond the rank-one reduction used in~\Cref{cor:schatten-rank-one}? In particular, can one extend the $p$-dependent lower bounds in~\cite{linguyenwoodruffon2014} to adaptive sketches? What is the correct TF complexity of singular-vector or nullspace estimation?
Moreover, to what extent can analogous impossibility results be established for regularized least squares problems or other structured inverse problems? We hope that the present work provides a starting point for these deeper questions in TF linear algebra.

\section*{Acknowledgments}
D.~H.~thanks the Simons Institute for the Theory of Computing for hosting her as a research fellow in the Fall 2025 program on Complexity and Linear Algebra. She also wishes to acknowledge many useful conversations with Cameron and Chris Musco on lower bound techniques and Pinsker's inequality. A.~T.~has been supported by the Defense Advanced Research Projects Agency (DARPA) through The Right Space (TRS) Disruption Opportunity (DARPA-PA-24-04-07). The views, findings, and conclusions expressed in this paper are those of the authors and do not necessarily reflect the official policy or position of DARPA, the U.S.~Department of Defense, or the U.S.~Government. A.~T.~has also been supported by the National Science Foundation CAREER grant DMS-2045646.

\footnotesize
\bibliographystyle{abbrv}
\bibliography{tf}

\end{document}